\documentclass[12pt]{amsart}
\usepackage{amssymb,amsthm,amsmath,a4}

\title{A remainder estimate for Weyl's law on Liouville
tori}
\author[H. Lapointe]{Hugues Lapointe}\thanks{Research was
supported by the Canada Graduate Scholarships Program.}
\address{D\'e\-par\-te\-ment de math\'ematiques et de
sta\-tistique, Univer\-sit\'e de Mont\-r\'eal CP 6128 succ
Centre-Ville, Mont\-r\'eal QC  H3C 3J7, Canada.}
\email{lapointe@dms.umontreal.ca}

\newcommand{\Area}{{\rm {Area}}}
\numberwithin{equation}{subsection}
\theoremstyle{definition}
\newtheorem{definition}[equation]{Definition}
\newtheorem{remark}[equation]{Remark}
\theoremstyle{plain}
\newtheorem{lemma}[equation]{Lemma}
\newtheorem{theorem}[equation]{Theorem}

\newtheorem{prop}[equation]{Proposition}

\begin{document}

\begin{abstract}
The paper is concerned with the asymptotic distribution of Laplace eigenvalues on Liouville tori.
Liouville metrics are the largest known class of integrable metrics on two-dimensional tori; they contain flat metrics and metrics of
revolution as special cases. Using separation of variables, we reduce the eigenvalue counting problem to the problem of counting lattice points in certain
planar domains. This allows us to improve the remainder estimate in Weyl's law on a large class of Liouville tori. For flat metrics, such an estimate has been known for more than a century due to classical results of  W.~Sierpi\'nski and J.G.~van der Corput. Our proof combines the method of Y.~Colin de Verdi\`ere, who proved an analogous result
for metrics of revolution on a sphere,  with the techniques developed by P.~Bleher, D.~Kosygin, A.~Minasov and Y.~Sinai in their study of the almost periodic properties of the remainder in Weyl's law on Liouville tori.

\end{abstract}


\maketitle

\section{Introduction and main results}
\subsection{Liouville tori}
\label{liouville}
A Liouville torus $T= \mathbb R^2 / (a_1 \mathbb Z \oplus a_2 \mathbb Z)$ is a two-dimensional torus with the metric
\begin{equation}
ds^2 = (U_1(q_1) - U_2(q_2))(dq_1^2 + dq_2^2),
\end{equation}
where $U_1(q_1) > U_2(q_2) > 0$ are smooth periodic functions on $\mathbb R$, satisfying $U_i(q_i+a_i) = U_i(q_i)$ for all $q_i \in \mathbb R$, $i=1,2$.
For simplicity, we assume that $a_1=a_2=1$, but all the proofs work for arbitrary $a_i>0$.

\begin{definition}
\label{condition11}
Let $\Omega$ be the set of pairs of functions $(U_1,U_2)$, satisfying the following conditions for $i=1,2$:
\begin{enumerate}
\item $U_i \in C^{\infty}(\mathbb{R})$.
\item $U_i(q+1) = U_i(q)$ for all $q \in \mathbb R$.
\item The function $U_i$ has exactly one minimum and one maximum in $[0,1)$, both nondegenerate.
\item $U_1(q_1) > U_2(q_2)$ for any $q_1,q_2 \in \mathbb R$.
\end{enumerate}
\end{definition}

If a Liouville torus $T$
admits a finite group of translations $G$ leaving the metric invariant, we may consider the quotient $T/G$.  Such tori are called infra-Liouville and were considered in
\cite{3}. Our results hold for infra-Liouville tori under some additional assumptions, as shown in section \ref{section24}.

\subsection{Weyl's law}
The Laplacian on $T$ 
is given by
\[
\Delta = - \frac{1}{U_1(q_1) - U_2(q_2)} \left(\frac{\partial^2}{\partial q_1^2} + \frac{\partial^2}{\partial q_2^2}\right).
\]
Its spectrum consists of an infinite sequence of eigenvalues
$$
0 =  \lambda_0 <  \lambda_1 \leq  \lambda_2 \leq  \lambda_3 \leq ... \leq \lambda_n \leq ... \rightarrow \infty
$$
of finite multiplicity. The corresponding eigenfunctions are smooth and form a basis in $L^2(T)$.
The spectral counting funtion is defined by
$$N(\lambda) = \#(\{j | \sqrt{\lambda_j} \leq \lambda\}),$$
where each eigenvalue is counted with its multiplicity. According to Weyl's law (see \cite{11,12}), the following asymptotic formula holds on any compact surface:
\[
R(\lambda) = N(\lambda) - \frac{\lambda^2}{4\pi} \Area(T)  = O(\lambda)
\]
The function $R(\lambda)$ is called the remainder term. The estimate $O(\lambda)$ is sharp and attained on a round sphere, since the spherical harmonics have high multiplicities.
However, it can be improved under certain assumptions on the surface.
For example, if the set of directions of periodic geodesics has Liouville measure zero in the unit cotangent bundle, it is shown in \cite{13} that the remainder is of order $o(\lambda)$.
This result has been extended to surfaces with boundary in \cite{15}.
Also, $R(\lambda) = O \left(\frac{\lambda}{\log \lambda}\right)$ for negatively curved surfaces (see \cite{14}), and
it is conjectured that for generic surfaces of negative curvature $R(\lambda) = O(\lambda^{\epsilon})$ for any $\epsilon > 0$ (see \cite{10}).

For a flat square torus, $R(\lambda)$ is the difference between the number
of points of $\mathbb Z^2$ lying inside a circle of radius $\lambda$ in $\mathbb R^2$ and the area of the disk bounded by this circle.
Gauss's bound was $O(\lambda)$ and Sierpi\'nski improved this estimate to $O(\lambda^{2/3})$ in \cite{16}.
This result can be obtained by smoothing the characteristic function of the domain and applying the Poisson summation formula (see \cite{17}).
A generalization of this method, applied to higher dimensional cases, is found in \cite{5}.
It has been shown recently (see \cite{8,7}) that the remainder in the circle problem is of order $O(\lambda^{\frac{131}{208}}(\log \lambda)^{\frac{18627}{8320}})$.
But we are still far from proving Hardy's conjecture that the bound $O(\lambda^{\frac{1}{2}+\epsilon})$ should hold for any $\epsilon > 0$, which would be optimal on the polynomial scale.

The upper bound on $R(\lambda)$ can also be improved for certain surfaces whose geodesic flow is completely integrable.
Colin de Verdi\`ere showed in \cite{4} that, for a generic
convex sphere of revolution, the remainder is of order $O(\lambda^{2/3})$.
\subsection{Main results}
On a nondegenerate Liouville torus $T$,
we estimate $R(\lambda)$ by representing the eigenvalues as lattice
points in a planar domain. This correspondence, described in section \ref{section23} and proved in \cite{1}, relies on
the separation of variables and the asymptotic analysis of Sturm-Liouville problems.
Certain techniques used to count points inside homothetic domains are then applied
to bound $R(\lambda)$. The conditions required for a Liouville torus to qualify as nondegenerate will be explained in section \ref{generic}.

Our main results are as follows.
\begin{theorem}
\label{thm1} The spectral counting function of a nondegenerate Liouville torus admits the following bound on its remainder term:
\[
R(\lambda) = O(\lambda^{2/3}).
\]
\end{theorem}

\begin{theorem}
\label{thm2} The set of nondegenerate metrics is dense in the set $\Omega$ (see Definition \ref{condition11}) in the Whitney $C^{\infty}$--topology.
\end{theorem}

In sections \ref{section4} and \ref{section24}, we show that analogous results also hold for nondegenerate tori of revolution
and infra-Liouville tori.

\begin{remark}
Nondegenerate metrics actually form a set of first Baire category, hence they are not generic.
This can be deduced from the results of sections $\ref{density}$ and $\ref{limits}$.
\end{remark}

\begin{remark}
It is conjectured in \cite{10} that the remainder term of any Liouville torus is of order $O(\lambda^{\frac{1}{2}+\epsilon})$ for $\epsilon > 0$ arbitrarily small.
Note that in the particular case of a flat square torus, it is equivalent to Hardy's conjecture for the Gauss's circle problem.
\end{remark}


\section{Eigenvalues of a Liouville torus}
In this section we review the known results regarding the geodesic flow and the eigenvalues of Liouville tori, obtained in \cite{2},\cite{6} and \cite{1}.
\subsection{Integrability of the geodesic flow}
\label{section21}

We study the geodesic flow on the cotangent bundle of $T$ by introducing the Hamiltonian $H(p,q) : \textbf{T}^{*}(T) \rightarrow \mathbb R$,
\[
H(p,q) = \frac{1}{U_1(q_1) - U_2(q_2)} (p_1^2 + p_2^2)
\]
The hamiltonian system defined by $H(p,q)$ is integrable since it has the following additional first integral
\[
S(p,q) = \frac{U_2(q_2)}{U_1(q_1) - U_2(q_2)}p_1^2 + \frac{U_1(q_1)}{U_1(q_1) - U_2(q_2)}p_2^2
\]
The Poisson bracket $\{H,S\}$ is identically zero so that the integrals $H$ and $S$ are in involution.
On a fixed energy level set we can write $H(p,q) = L^2$ and $S(p,q) = cL^2$ for two constants $L$ and $c$. We define
\[
c_1 = \max_{0 \leq x \leq 1} U_1(x) = U_1(M_1), \qquad c_2 = \min_{0 \leq x \leq 1} U_1(x) = U_1(m_1),
\]
\[
c_3 = \max_{0 \leq x \leq 1} U_2(x) = U_2(M_2), \qquad c_4 = \min_{0 \leq x \leq 1} U_2(x) = U_2(m_2).
\]
and assume that the second derivative of the $U_i$ functions does not vanish at their respective critical points. Since
\[
p_1^2 = (U_1(q_1) - c)L^2, \qquad p_2^2 = (c - U_2(q_2))L^2
\]
the action variables are given by
\[
I_1(L,c)=
\begin{cases}
L \int_{0}^{1} (U_1(q_1) - c)^{1/2} dq_1 & \text{if } c_4 \leq c \leq c_2\\
L \int_{U_1(q_1) \geq c} (U_1(q_1) - c)^{1/2} dq_1 & \text{if } c_2 \leq c \leq c_1\\
\end{cases}
\]
\[
I_2(L,c)=
\begin{cases}
L \int_{U_2(q_2) \leq c} (c - U_2(q_2))^{1/2} dq_2 & \text{if } c_4 \leq c \leq c_3\\
L \int_{0}^{1} (c - U_2(q_2))^{1/2} dq_2& \text{if } c_3 \leq c \leq c_1\\
\end{cases}
\]
We shall write $F_1(c) = I_1(1,c)$ and $F_2(c) = I_2(1,c)$.
These functions define a curve $\gamma = (F_1(c),F_2(c))$ in the plane for $c \in [c_4, c_1]$.
Theorem 3.2 of \cite{1} gives some informations on the functions $F_i$, which we repeat here,
\begin{theorem}
The functions $F_i$ satisfy the following properties:
\begin{enumerate}
\label{asympt1}
\item $F_1(c)$ is a continuous function that is strictly decreasing. Moreover, \linebreak $F_1(c) \in C^{\infty}([c_4,c_2) \cup (c_2, c_1])$
and $F_1(c_1) = 0, {F_1}'(c_1) = -\pi (-2U_{1}^{''}(M_1))^{-1/2}$.
\item $F_2(c)$ is a continuous function that is strictly increasing. Moreover, \linebreak $F_2(c) \in C^{\infty}([c_4,c_3) \cup (c_3, c_1])$
and $F_2(c_4) = 0, {F_2}'(c_4) = \pi (2U_{2}^{''}(m_2))^{-1/2}$.
\item Close to their critical points, the derivatives of the $F_i$ functions have the following asymptotics,
\[
\lim_{c \rightarrow c_{2}} \frac{dF_1(c)}{dc} \frac{1}{\log|c-c_{2}|} = (\frac{1}{2}U_{1}^{''}(m_1))^{-1/2}
\]
\[
\lim_{c \rightarrow c_{3}} \frac{dF_2(c)}{dc} \frac{1}{\log|c-c_{3}|} = -(-\frac{1}{2}U_{2}^{''}(M_2))^{-1/2}
\]
Asymptotics for higher derivatives of $F_i$ can be found by differentiating the previous expressions.
\item For $c \in [c_4,c_1]$, the derivatives $dF_i/dc$ are different from zero.
\end{enumerate}
\end{theorem}

There exists a function $G(\alpha)$ such that $\gamma$ is defined in polar coordinates by $\rho = G(\alpha), 0 \leq \alpha \leq \pi /2$,
where
\[
\begin{cases}
G(\alpha) = (F_1(c(\alpha))^2+F_2(c(\alpha))^2)^{1/2}\\
\tan (\alpha) = F_2(c(\alpha))/F_1(c(\alpha))\\
\end{cases}
\]
We set
\[
\alpha_0 = \alpha(c_4) = 0,
\]
\[
\alpha_1 = \alpha(c_3) = \arctan \frac{F_2(c_3)}{F_1(c_3)},
\]
\[
\alpha_2 = \alpha(c_2) = \arctan \frac{F_2(c_2)}{F_1(c_2)},
\]
\[
\alpha_3 = \alpha(c_1) = \frac{\pi}{2}.
\]
The function $G(\alpha)$ is studied in Theorem 6.3 of \cite{1}. We repeat here their results,
\begin{enumerate}
\item $G(\alpha) \in C^1([0, \frac{\pi}{2}])$. The tangent to $\gamma$ at the point with
angular coordinates $\alpha_1$ is vertical and the tangent to $\gamma$ at the point with angular coordinates $\alpha_2$ is horizontal.
\item $G(\alpha) \in C^{\infty}([0, \alpha_1) \cup (\alpha_1, \alpha_2) \cup (\alpha_2,\frac{\pi}{2}])$
\item We have $G^{(0)}(+0) = G(0) > 0$, $G^{(0)}(\frac{\pi}{2}-0) = G(\frac{\pi}{2}) > 0$ and
$G^{(1)}(+0) \neq 0$, $G^{(1)}(\frac{\pi}{2}-0)\neq 0$.
\item Close to the singularities at the angles $\alpha_1$ and $\alpha_2$, the second derivative of $G$ has the
following asymptotics,
\[
\lim_{\alpha \rightarrow \alpha_i} \frac{d^2G(\alpha)}{d\alpha^2} (\alpha - \alpha_i)(\log|\alpha - \alpha_i|)^2 = \text{const}(i) , \qquad i = 1,2
\]
\item We have the following inequalities
\[
M \geq G(\alpha) \geq m > 0, \qquad \text{ for } 0 \leq \alpha \leq \frac{\pi}{2}
\]
\[
\frac{dG(\alpha)}{d\alpha}\leq K, \qquad \text{ for }  0 \leq \alpha \leq \frac{\pi}{2}
\]
\end{enumerate}
Note that this implies the existence of a constant $\delta$ such that, for $\epsilon > 0$,
\begin{equation}
\label{varib}
\text{dist}(\gamma,(1+\epsilon)\gamma) \geq \delta \epsilon
\end{equation}
where $\text{dist}(X,Y)$ is the Euclidean distance between two subsets $X$,$Y$ of $\mathbb R^2$.

\subsection{Nondegeneracy conditions}
\label{generic} In what follows, we will require some quantities to
be irrational and difficult to approximate using rational numbers
(see \cite[section 11]{1}).
\begin{definition}
\label{typicalnum}
A real number $\alpha$ is typical if there exists $\tau > 0$ such that
\[
\left| k_1 + \alpha k_2 \right| \geq \frac{\delta(\alpha)}{|k_2| \log(1+|k_2|)^{\tau}}, \qquad \forall k_1,k_2 \in \mathbb Z, k_2 \neq 0
\]
\end{definition}
Given any $\tau > 1$, almost all real numbers satisfy the above inequality for some constant $\delta(\alpha)$.
Indeed, for $k,n \in \mathbb N$, consider the set
\[
S_{k,n}=\{x \in [0,1]| \exists l \in \mathbb Z, |l-xk|<(kn \log(1+k)^{\tau})^{-1} \}
\]
The set of typical numbers
associated to the exponent $\tau$ is the complement of $\bigcap_{n=1}^{+\infty}
\bigcup_{k=1}^{+\infty} S_{k,n}$.
The Lebesgue measure of $S_{k,n}$ is lower than $2(nk \log(1+k)^{\tau})^{-1}$, and
the measure of $\bigcup_{k=1}^{+\infty} S_{k,n}$ is lower than
\[
\frac{1}{n} \sum_{k=1}^{+\infty} \frac{2}{k \log(1+k)^{\tau}} = \frac{C}{n}
\]
for $\tau > 1$ and
all $n$. Thus the set $\bigcap_{n=1}^{+\infty} \bigcup_{k=1}^{+\infty} S_{k,n}$ has measure zero
and almost every real number is typical for the exponent $\tau > 1$. However, the set of typical
numbers is also of first Baire category since it is the denumerable union of nowhere dense closed subsets.

Given a metric represented by $(U_1,U_2)$,
the curvature $\kappa(c)$ of $\gamma$ at a point $(F_1(c),F_2(c))$ is given by $\frac{F_2^{''}F_1^{'} - F_1^{''}F_2^{'}}{((F_1^{'})^2+(F_2^{'})^2)^{3/2}}$.
The only points where the curvature diverges are those corresponding to the singularities at $c_2$ and $c_3$.
Note that on the interval $c \in (c_3,c_2)$,
\[
8\,(F_2^{''}F_1^{'} - F_1^{''}F_2^{'})(c) = \int_{0}^1 (c -
U_2(q_2))^{-3/2} dq_2 \int_{0}^1 (U_1(q_1) - c)^{-1/2} dq_1
\]
\[
+ \int_{0}^1 (U_1(q_1) - c)^{-3/2} dq_1 \int_{0}^1 (c - U_2(q_2))^{-1/2} dq_2
\]
so $\kappa(c)$ cannot vanish there.
\begin{definition}
\label{genericity}
The metric $ds^2=(U_1(q_1) - U_2(q_2))(dq_1^2 + dq_2^2)$ on $T$ is said to be nondegenerate if $(U_1,U_2) \in \Omega$
and the following conditions hold.
\begin{enumerate}
\item The curvature $\kappa(c)$ has a finite number of zeros on $[c_4,c_3) \cup (c_2,c_1]$, each of first order.
\item If $\kappa(\tilde{c})=0$, then ${F_2}'(\tilde{c})/{F_1}'(\tilde{c})$ is a typical number.
\item The numbers ${F_2}' (c_i) / {F_1}' (c_i)$ are typical for $i = 1,4$.
\[
{F_2}' (c_1) / {F_1}' (c_1) = \frac{ \int_{0}^1 (c_1 - U_2(q_2))^{-1/2} dq_2 }
{-2\pi (-2U_{1}^{''}(M_1))^{-1/2}}
\]
\[
{F_2}' (c_4) / {F_1}' (c_4) = \frac{ 2\pi (2U_{2}^{''}(m_2))^{-1/2} }
{ \int_{0}^1 (U_1(q_1) - c_4)^{-1/2} dq_1}
\]
\item The numbers $F_2 (c_i) / F_1 (c_i)$ are typical for $i = 2,3$.
\[
F_2 (c_2) / F_1 (c_2) = \frac{ \int_{0}^1 (c_2 - U_2(q_2))^{1/2} dq_2 }
{ \int_{0}^1 (U_1(q_1) - c_2)^{1/2} dq_1}
\]
\[
F_2 (c_3) / F_1 (c_3) = \frac{ \int_{0}^1 (c_3 - U_2(q_2))^{1/2} dq_2 }
{ \int_{0}^1 (U_1(q_1) - c_3)^{1/2} dq_1}
\]
\end{enumerate}
\end{definition}
Such conditions are also required in Theorem 3.1 of \cite{2}.

\subsection{Eigenvalues on a Liouville torus}
\label{section23}
The Laplace operator on $T$ has the form
\[
\Delta = - \frac{1}{U_1(q_1) - U_2(q_2)} \left(\frac{\partial^2}{\partial q_1^2} + \frac{\partial^2}{\partial q_2^2}\right)
\]
We associate to the first integral $S(p,q)$ another operator,
\[
\hat{S} = - \frac{U_2(q_2)}{U_1(q_1) - U_2(q_2)}\frac{\partial^2}{\partial q_1^2} - \frac{U_1(q_1)}{U_1(q_1) - U_2(q_2)}\frac{\partial^2}{\partial q_2^2}
\]
For a given pair of integers $m = (m_1,m_2)$, with $m_1 \geq 0$ and $m_2 \geq 0$,
there is a pair of eigenvalues $(E_{m}, \widetilde{E}_{m})$ of $(\Delta,\hat{S})$.
They correspond to solutions of the following periodic Sturm-Liouville problems, obtained
after separation of variables,
\begin{equation}
\label{sturm}
\begin{cases}
\Psi_1^{''} + (E_m U_1 - \widetilde{E}_m) \Psi_1 = 0 \\
\Psi_2^{''} + (\widetilde{E}_m - E_m U_2) \Psi_2 = 0
\end{cases}
\end{equation}
More precisely, $\widetilde{E}_m$ is such that $E_m$ is the $m_1$-th eigenvalue of the first equation
and the $m_2$-th eigenvalue of the second equation. Note that given $\widetilde{E}$,
the solutions $E$ form an increasing sequence in the first case and a decreasing sequence
in the second case of \eqref{sturm}.

We set $E_m = \lambda^2$ and $c = \widetilde{E}_m / E_m$.
Theorem 6.1 of \cite{1} says that, for $|m|$ sufficiently large, $(E_{m}, \widetilde{E}_{m})$ is the unique solution to the equation
\begin{equation}
\label{eq1}
\Phi_0(\lambda,c) + \Phi_1(\lambda,c) + \Phi_2(\lambda,c) = 2 \pi \left( \left[ \frac{m_1+1}{2} \right],\left[ \frac{m_2+1}{2}\right] \right)
\end{equation}
where $\Phi_0(\lambda,c) = \lambda(F_1(c), F_2(c))$ and $|\Phi_2(\lambda,c)| \leq \text{ Const }  \lambda^{-2/3} \log \lambda$ uniformly for $c \in [c_4, c_1]$.
The function $\Phi_1(\lambda,c)$ is of the form
\[
\Phi_1(\lambda,c) = (\phi_1(\lambda,c), \phi_2(\lambda,c))
\]
and the following bounds apply, depending on the location of $(m_1,m_2)$ in the plane,
\[
\begin{cases}
|\phi_1(\lambda,c) + (-1)^{m_1}\frac{\pi}{2}| \leq \text{Const}  \lambda^{-2/3} \log \lambda ,& \text{  for  } c_2 + \text{const} \lambda^{-2/3} \leq c \leq c_1\\
|\phi_1(\lambda,c)| \leq \text{Const}  \lambda^{-2/3} \log \lambda ,& \text{  for  } c_4 \leq c \leq c_2 - \text{const} \lambda^{-2/3}\\
|\phi_1(\lambda,c)| \leq \text{Const}  ,& \text{  in other cases}\\
\end{cases}
\]
\[
\begin{cases}
|\phi_2(\lambda,c)| \leq \text{Const}  \lambda^{-2/3} \log \lambda ,& \text{  for  } c_3
+ \text{const} \lambda^{-2/3} \leq c \leq c_1\\
|\phi_2(\lambda,c) + (-1)^{m_2}\frac{\pi}{2}| \leq \text{Const}  \lambda^{-2/3} \log \lambda ,& \text{  for  } c_4 \leq c \leq c_3 - \text{const} \lambda^{-2/3}\\
|\phi_2(\lambda,c)| \leq \text{Const} ,& \text{  in other cases}\\
\end{cases}
\]
Note that for $m_1, m_2 \geq 0$, each point of the form $(2 \pi k_1, 2 \pi k_2)$, $k_1, k_2 > 0$ can be written in four different ways as $2 \pi \left( \left[ \frac{m_1+1}{2} \right],\left[ \frac{m_2+1}{2}\right] \right)$,
the points lying on one axis in two, and the origin in a unique way.

The following domains will be used later
\[
A_i = \{ (\rho, \alpha) | \alpha_{i-1} \leq \alpha \leq \alpha_{i}, 0 \leq \rho \leq G(\alpha) \},\qquad i = 1,2,3
\]
\[
A = A_1 \cup A_2 \cup A_3
\]
We consider $A_i$ and $A$ as subsets of $\mathbb{R}^2$, through the mapping $(x,y)=(\rho \cos \alpha,\rho \sin \alpha)$.

Also, for $a=(a_1,a_2) \in \mathbb{R}^2$, we define the translated two-dimensional lattice
\[
\Gamma_a = \{\left(2 \pi k_1 +a_1 , 2 \pi k_2 + a_2 \right) | (k_1,k_2) \in \mathbb Z^2 \}
\]
and, for $D$ a subset of $\mathbb{R}^2$ with $rD=\{ (x,y) \in \mathbb{R}^2 | (r^{-1}x,r^{-1}y) \in D\}$, let
\[
N_a(D,r)  =   \sharp(\Gamma_a \cap r\mathring{D}) +  \frac{1}{2} \sharp (\Gamma_a \cap r \partial D) \\
\]
\[
 +   \sharp(\Gamma_{-a} \cap r\mathring{D}) + \frac{1}{2} \sharp (\Gamma_{-a} \cap r \partial D)
\]
where $\mathring{D}$ is the interior of $D$ and $\partial D$ its boundary.
The function $N_a(D,r)$ counts the number of points in $rD$ of two opposite translates of $2\pi \mathbb{Z}^2$, giving
a weight of $\frac{1}{2}$ to those lying on the boundary of $rD$.

\section{Proofs of the main results}
In section \ref{section31}, we explain the method used to count points in the domains $r A_i$, $i=1,2,3$,
and obtain several bounds on the Fourier transform of domains bounded by the curve $\gamma$.
Additional results required for lattice counting are collected in section \ref{straight}.
The proofs of Theorems \ref{thm1} and \ref{thm2} are contained in sections \ref{boundremterm} and \ref{density}, respectively.
We show that nondegenerate metrics are not generic in section \ref{limits}.

\subsection{Regularization of the counting function}
\label{section31}
We follow the approach used in section 16 of \cite{1}.
Let $\psi$ be a positive function such that $\psi \in C_{0}^{\infty}(\mathbb{R})$, $\text{ supp }
\psi \subset (-1,1)$,
$\int_{\mathbb R^2} \psi(\sqrt{x^2+y^2}) dx dy= 1$ and $\psi \equiv 1$ in a neighbourhood of $0$. We will use the following cut-off function, $\Psi_{\epsilon}(x,y) = \epsilon^{-2} \psi(\epsilon^{-1}\sqrt{x^2+y^2})$.
We write $\Psi(x,y)$ for $\Psi_1(x,y)$.
Let $\chi_{D}$ be the characteristic function of the domain $D$,
\[
\chi_{D}(x,y) =
\begin{cases}
1 & \text{if $(x,y) \in D \setminus \partial D$}\\
\frac{1}{2} & \text{if $(x,y) \in \partial D$}\\
0 & \text{else}
\end{cases}
\]
and $\tilde{N}_a(D,r)$ be the regularized function,
\[
\tilde{N}_a(D,r) = \sum_{k \in 2 \pi \mathbb Z^2 +a} (\Psi_{r^{-4/3}} \ast \chi_{D}) (r^{-1} k)
+
\sum_{k \in 2 \pi \mathbb Z^2 -a} (\Psi_{r^{-4/3}} \ast \chi_{D}) (r^{-1} k)
\]
which is an approximation to $N_a(D,r)$.
Note that the points for which $\chi_{D}(r^{-1} k) \neq (\Psi_{r^{-4/3}} \ast \chi_{D}) (r^{-1} k)$
lie at distance of order $O(r^{-1/3})$ from $r \partial D$.

Using the Poisson summation formula we get
\[
\tilde{N}_a(D,r) =
\frac{r^2}{2 \pi^2} \sum_{k \in \mathbb Z^2} \cos(\langle a,k \rangle) \hat{\chi}_{D}(rk) \hat{\Psi}(r^{-1/3}k)
\]
Since $\hat{\Psi}(0) = 1$ and $\hat{\chi}_{D}(0) = \Area(D)$, the term corresponding to $k=0$ is
$\frac{\Area(D)}{2 \pi^2} r^2$.
If $k \neq 0$, Stokes' formula gives,
\[
\hat{\chi}_{D}(rk) = \int_{D} e^{-ir(x k_1 + y k_2)} dx dy = \oint_{\partial D} \frac{1}{ir} \frac{e^{-ir(x k_1 + y k_2)}}{(k_1^2 + k_2^2)}(k_2 dx - k_1 dy)
\]
We have that $\hat{\Psi}(r^{-1/3}k)$ depends only on $r^{-1/3}|k|$ and is rapidly decreasing as its argument tends to infinity.
However, the only bound needed is
\[
|\hat{\Psi}(r^{-1/3}k)| \leq \frac{C}{(1+r^{-1/3}|k|)^\nu}
\]
for a fixed $\nu > 2$.
Under some restrictions on $\partial D$, we want to obtain
\begin{equation}
\frac{r}{2 \pi^2}
\sum_{k \neq 0}
\cos(\langle a,k \rangle) \hat{\Psi}(r^{-1/3}k)
\oint_{\partial D} \frac{e^{-ir(x k_1 + y k_2)}}{(k_1^2 + k_2^2)}(k_2 dx - k_1 dy)
= O(r^{2/3})
\end{equation}
so that
\[
\tilde{N}_a(D,r) = \frac{\Area(D)}{2 \pi^2} r^2 + O(r^{2/3})
\]
and
\[
\tilde{N}_a(D,r-\varkappa r^{-1/3}) \leq N_a(D,r) \leq \tilde{N}_a(D,r+\varkappa r^{-1/3})
\]
for $\varkappa$ sufficiently large and $D$ star-shaped with respect to the origin, implies
\[
N_a(D,r) = \frac{\Area(D)}{2 \pi^2} r^2 + O(r^{2/3})
\]
This estimate is required in the proof of Theorem \ref{thm1}, where the region $D$ considered is either $A$ or $A_i$, with $i = 1,2,3$.

Note that the choice $\epsilon = r^{-4/3}$ is optimal with this method,
since taking the exponent $-4/3 + \delta$ instead would give at best a remainder of order
\[
\frac{\Area(D)}{2 \pi^2} \left( (r+\varkappa r^{-1/3+\delta})^2 - (r-\varkappa r^{-1/3+\delta})^2 \right) + O(r^{2/3 - \delta/2})
\]
\[
= O(r^{2/3}(r^{\delta} +r^{-\delta/2}))
\]

We separate $\partial D$ in several pieces, and consider each of them separately.
In particular, suppose $\partial D$ contains some parts of $\gamma$.
We reparametrize the curve $\gamma$ from $(F_1(c),F_2(c))$ to $(t,f(t))$, with $t \in [0,F_{1}(c_{4})]$.
We use a partition to study $f$ on distinct intervals $[t_j,t_{j+1}]$, where $t_j$ is an increasing sequence with
$t_0=0$ and $t_{m+1}=F_{1}(c_{4})$,
\[
[0,F_{1}(c_4)] = \bigcup_{j=0}^{m} [t_j,t_{j+1}]
\]
and require that
$f(t) \in C^{0}([t_j,t_{j+1}])$, $f(t) \in C^{\infty}((t_j,t_{j+1}))$ and $f^{''}(t) \neq 0$ for $t \in (t_j,t_{j+1})$.
The function $f(t)$ is singular at $F_1(c_2)$ and $F_1(c_3)$, since
\[
f'(t) = \frac{F_{2}^{'}(c(t))}{F_{1}^{'}(c(t))} \text{ and } f''(t) = \frac{(F_{2}^{''}F_1^{'}-F_1^{''}F_2^{'})(c(t))}{F_{1}^{'}(c(t))^3}
\]
Here, the derivatives of the $F_i$ functions are taken with respect to the $c$ variable.
We assume that the nondegeneracy conditions of Definition \ref{genericity} are fulfilled by the metric.
The points $t_j$ will be the singular points of $f(t)$, the first order zeros of $f''(t)$, where $f'(t_j)$
is a typical number, and the ends of $\gamma$, $0$ and $F_{1}(c_4)$.

Given a piece of $\gamma$, corresponding to the interval $[t_j,t_{j+1}]$,
let
\[
E = \{(k_1,k_2) \in \mathbb Z^2 | (k_1 + f'(s_k) k_2) = 0,  s_k \in (t_j,t_{j+1}) \}
\]
Since $f'(t)$ is monotone on $(t_j,t_{j+1})$, $s_k$ is well defined for each $k \in E$ different from zero.
We treat the cases of $k \notin E$ and $k \in E$ separately.

\begin{theorem}
\label{lma}
The following bound holds
\[
\sum_{k \notin E} \frac{r}{(1+r^{-1/3}|k|)^\nu}
\left| \int_{t_j}^{t_{j+1}} \frac{e^{-ir(t k_1 + f(t) k_2)}}{(k_1^2 + k_2^2)}
(k_2  - k_1 f'(t)) dt \right|
\]
\[
= O(r^{2/3})
\]
\end{theorem}

\begin{proof}
For $k \notin E$, we can integrate by parts,
\[
\int_{t_j}^{t_{j+1}} \frac{e^{-ir(t k_1 + f(t) k_2)}}{(k_1^2 + k_2^2)}(k_2 - k_1 f'(t)) dt
\]
\begin{equation}
\label{byparts}
= \left. \frac{1}{-ir(k_1^2 + k_2^2)} \frac{(k_2 - k_1 f'(t))}{(k_1 + k_2 f'(t))} e^{-ir(t k_1 + f(t) k_2)} \right|_{t_j}^{t_{j+1}}
\end{equation}
\[
+ \frac{1}{-ir}\int_{t_j}^{t_{j+1}} \frac{f''(t)}{(k_1 + k_2 f'(t))^2} e^{-ir(t k_1 + f(t) k_2)} dt
\]
if both $f'(t_j)$ and $f'(t_{j+1})$ are typical numbers.
The contribution of the first term on the right hand side of \eqref{byparts} will be bounded by
\begin{equation}
\label{eandnote}
\frac{1}{r(k_1^2 + k_2^2)} \left( \left|  \frac{k_2 - k_1 f'(t_{j+1})}{k_1 + k_2 f'(t_{j+1})} \right| +
\left| \frac{k_2 - k_1 f'(t_{j})}{k_1 + k_2 f'(t_{j})} \right| \right)
\end{equation}
Since $f^{''}(t)$ is of constant sign on $(t_j,t_{j+1})$,
we can integrate the absolute value of the integrand and bound the last term of \eqref{byparts} by
\begin{equation}
\frac{1}{r}
\left| \frac{f'(t_{j})-f'(t_{j+1})}
{(k_1+k_2 f'(t_{j}))(k_1+k_2 f'(t_{j+1}))} \right|
\end{equation}
Summing over $k \notin E$, after multiplying each term by the weight $\frac{r}{(1+r^{-1/3}|k|)^\nu}$, we get a contribution of maximum order
\[
\sum_{k \in \mathbb{Z}^2} \frac{1}{(1+r^{-1/3}|k|)^\nu} + \sum_{n=1}^{+\infty} \frac{\log(1+n)^{\tau}}{(1+r^{-1/3}n)^\nu}
\]
\[
=O(r^{2/3})
\]

If $t_j$ or $t_{j+1}$ is equal to $F_1(c_2)$ or $F_1(c_3)$, we must be careful with the integration by parts because $f'(F_1(c_3))$ diverges and $f'(F_1(c_2))=0$.
We study only the case of $t_j=F_1(c_3)$ since the others are equivalent.
The difficulty appears when handling the pairs $(k_1,k_2)$ for which $k_2=0$.
We know that the following asymptotic holds for $t$ near enough $t_j$,
\[
M \log(t-t_j) < f'(t) < m \log(t-t_j)
\]
We deduce that for $\epsilon$ small enough,
\[
\left| \int_{t_j}^{t_j+\epsilon} \frac{e^{-ir k_1 t}}{k_1} f'(t) dt \right| \leq  -C \frac{\epsilon \log(\epsilon)}{|k_1|}
\]
and, using integration by parts, we also have
\[
\left| \int_{t_{j}+\epsilon}^{t_{j+1}} \frac{e^{-ir k_1 t}}{k_1} f'(t) dt \right| \leq -C \frac{\log(\epsilon)}{r k_1^2}
\]
We choose $\epsilon=(r |k_1|)^{-1}$ so that the contribution of these terms is bounded by
\[
\sum_{n=1}^{+\infty} \frac{C \log(r n)}{ n^2 (1+r^{-1/3} n)^\nu} = O(\log(r) )
\]
\end{proof}

\begin{theorem}
\label{lmc}
The following bound holds
\[
\sum_{\substack{k \in E \\ k \neq 0}} \frac{r}{(1+r^{-1/3}|k|)^\nu}
\left| \int_{t_j}^{t_{j+1}} \frac{e^{-ir(t k_1 + f(t) k_2)}}{(k_1^2 + k_2^2)}
(k_2  - k_1 f'(t)) dt \right|
\]
\[
= O(r^{2/3})
\]
\end{theorem}
\begin{proof}
For $k \in E$, if $k_1 + k_2 f'(s_k) = 0$, by definition $s_k \in (t_j,t_{j+1})$.
In these cases we must use an approach similar to the stationary phase method to get an asymptotic as $r \rightarrow +\infty$. The stationary phase formula
works for the integration of functions in $C^{\infty}_{0}(\mathbb R)$, so that
we cannot apply it directly here. Instead,
we integrate separately on subintervals of $(t_j,t_{j+1})$ depending on $r$.
The main term will correspond to the one given by the stationary phase but the error term will be easier to estimate.
We integrate by parts on $(t_j,s_k)$ and $(s_k,t_{j+1})$,
\begin{equation}
\label{lintegrale}
-ir \int_{t_j}^{t_{j+1}} \frac{e^{-ir(t k_1 + f(t) k_2)}}{(k_1^2 + k_2^2)}(k_2 - k_1 f'(t))dt
\end{equation}
\[
= \lim_{\epsilon \rightarrow 0^{+}} \left( \int_{t_j}^{s_k-\epsilon}  + \int_{s_k+\epsilon}^{t_{j+1}} \right)
e^{-ir(t k_1 + f(t) k_2)}\frac{f''(t)}{(k_1 + k_2 f'(t))^2}dt
\]
\[
+ \frac{ e^{-ir(t k_1 + f(t) k_2)}}{(k_1^2+k_2^2)} \frac{k_2 - k_1 f'(t)}{k_1+k_2 f'(t)}  \left( \left. \right|_{t_j}^{s_k-\epsilon} +
\left. \right|_{s_k+\epsilon}^{t_{j+1}} \right)
\]
Note that
\[
\frac{d}{dt} \left(\frac{k_2 - k_1 f'(t)}{k_1 + k_2 f'(t)}\right) = -\frac{(k_1^2+k_2^2)f''(t)}{(k_1+k_2 f'(t))^2}
\]
so we can write \eqref{lintegrale} as
\begin{multline}
\label{parts}
\lim_{\epsilon \rightarrow 0^{+}} \left( \int_{t_j}^{s_k-\epsilon}  + \int_{s_k+\epsilon}^{t_{j+1}} \right)
(e^{-ir(t k_1 + f(t) k_2)}-e^{-ir(s_k k_1 + f(s_k) k_2)})\frac{f''(t)}{(k_1 + k_2 f'(t))^2}dt
\\
+ \frac{ (e^{-ir(t k_1 + f(t) k_2)}-e^{-ir(s_k k_1 + f(s_k) k_2)})}{(k_1^2+k_2^2)} \frac{k_2 - k_1 f'(t)}{k_1+k_2 f'(t)}
\left( \left. \right|_{t_j}^{s_k-\epsilon} +
\left. \right|_{s_k+\epsilon}^{t_{j+1}} \right)
\end{multline}
Since $k_1+k_2 f'(t)$ has only a first order zero at $s_k$,
\[
\lim_{\epsilon \rightarrow 0} \frac{e^{-ir((s_k + \epsilon) k_1 + f(s_k + \epsilon) k_2)}
-e^{-ir(s_k k_1 + f(s_k) k_2)}}{k_1+k_2 f'(s_k + \epsilon)}
= 0
\]
and we can put $\epsilon = 0$ in the second term of \eqref{parts}.
Its contribution will then be bounded by \eqref{eandnote} and we can proceed as in the case
of $k \notin E$.

It now suffices to bound the first term of \eqref{parts}, which we write, after dividing by $e^{-ir(s_k k_1 + f(s_k) k_2)}$, as
\[
\int_{t_j}^{t_{j+1}}
\frac{(e^{-ir k_2 f''(s_k) \frac{(t-s_k)^2}{2} h_{k}(t)}-1)}{(t-s_k)^2}\frac{f''(t)}{(k_2 g_{k}(t)
f''(s_k))^2}dt
\]
for $tk_1 + f(t) k_2 = (s_k k_1 + f(s_k) k_2) + k_2 f''(s_k) \frac{(t-s_k)^2}{2} h_{k}(t)$ and $k_1 + k_2 f'(t) = k_2 g_{k}(t)
f''(s_k) (t-s_k)$.
Note that the integrand will be a continuous function, since it has a limit as $t \rightarrow s_k$.
We also have
\[
h_k(t)=2\frac{f(t)-f(s_k)-f'(s_k)(t-s_k)}{f''(s_k)(t-s_k)^2}
 = 2\frac{\int_{s_k}^t (t-u)f''(u)du}{f''(s_k)(t-s_k)^2}
\]
and
\[
g_{k}(t) = \frac{1}{f''(s_k)(t-s_k)}\int_{s_k}^t f''(u)du
\]

Our assumptions on $f''(t)$ at the points $t_j$ and $t_{j+1}$ imply the existence of
$H$ such that $0 < 1/H < g_k,h_k < H$
in the intervals $t \in [s_k-\alpha_k,s_k +\alpha_k]$,
for $k \in E$ and
\begin{equation}
\label{formulak}
\alpha_k = \frac{1}{2}\min \{ s_k - t_j, t_{j+1}-s_k \}
\end{equation}
A justification of this claim is contained in Lemma \ref{lmd}.
We define a function $B$ on $k \in E$,
\begin{equation}
B(k_1,k_2) = \alpha_k^{-2} |k_2 f''(s_k)|^{-1}
\end{equation}
It will be used to separate in two parts the sum in Theorem \ref{lmc},
depending on the value of $r$.

Using the estimate $|e^{z}-1| \leq e^C |z|$ for $|z| \leq C$, we deduce that for $r \geq B(k_1,k_2)$
\[
\left| \int_{s_k-|r k_2 f''(s_k)|^{-\frac{1}{2}}}^{s_k+|r k_2 f''(s_k)|^{-\frac{1}{2}}}
(e^{-ir k_2 f''(s_k) \frac{(t-s_k)^2}{2} h_{k}(t)}-1)\frac{f''(t)}{(k_1 + k_2 f'(t))^2}dt \right|
\]
\[
\leq e^{H}H^3 \int_{s_k-|r k_2 f''(s_k)|^{-\frac{1}{2}}}^{s_k+|r k_2 f''(s_k)|^{-\frac{1}{2}}}
\frac{r |k_2 f''(s_k)|(t-s_k)^2}{2} \frac{|f''(t)| dt}{(k_2 f''(s_k) (t - s_k))^2}
\]
\begin{equation}
\label{lmeeq1}
\leq e^H H^4  \frac{r^{1/2}}{|k_2^3 f''(s_k)|^{1/2} }
\end{equation}
since $(k_2 f''(s_k) (t - s_k))^2  \leq H^2 (k_1 + k_2 f'(t))^2$ and
\[
\int_{s_k-|r k_2 f''(s_k)|^{-\frac{1}{2}}}^{s_k+|r k_2 f''(s_k)|^{-\frac{1}{2}}} |f''(t)| dt
\]
\[
 = \left| f'(s_k + |r k_2 f''(s_k)|^{-\frac{1}{2}} ) - f'(s_k - |r k_2 f''(s_k)|^{-\frac{1}{2}} ) \right|
\]
\[
\leq 2 H |f''(s_k)|^{1/2} |r k_2|^{-1/2}
\]
Also
\[
\left| \int_{t_j}^{s_k-|r k_2 f''(s_k)|^{-\frac{1}{2}}}
 (e^{-ir(t k_1 + f(t) k_2)}-e^{-ir(s_k k_1 + f(s_k) k_2)})\frac{f''(t)}{(k_1 + k_2 f'(t))^2}dt \right|
\]
\[
+
\left|\int_{s_k+|r k_2 f''(s_k)|^{-\frac{1}{2}}}^{t_{j+1}}
 (e^{-ir(t k_1 + f(t) k_2)}-e^{-ir(s_k k_1 + f(s_k) k_2)})\frac{f''(t)}{(k_1 + k_2 f'(t))^2}dt \right|
\]
\[
\leq \left|  \frac{2}{k_2(k_1+k_2 f'(t))} \left( \left.  \right|_{t_j}^{s_k-|r k_2 f''(s_k)|^{-\frac{1}{2}}} +
\left. \right|_{s_k+|r k_2 f''(s_k)|^{-\frac{1}{2}}}^{t_{j+1}} \right) \right|
\]
\begin{equation}
\label{lmeeq2}
\leq 2 \left| \frac{f'(t_{j})-f'(t_{j+1})}{(k_1+k_2 f'(t_{j}))(k_1+k_2 f'(t_{j+1}))} \right| + 4 H \frac{r^{1/2}}{|k_2^3 f''(s_k)|^{1/2} }
\end{equation}
We have obtained a bound for large enough $r$, on each individual term of the summation over $k \in E$.
However, if we are to work with a fixed $r$, we must also have a bound for the terms such that $r < B(k_1,k_2)$.
This is equivalent to $\alpha_k < |r k_2 f''(s_k)|^{-1/2}$.
In these cases
\[
\left| \int_{s_k-\alpha_k}^{s_k+\alpha_k}
(e^{-ir k_2 f''(s_k) \frac{(t-s_k)^2}{2} h_k(t)}-1)\frac{f''(t)}{(k_1 + k_2 f'(t))^2}dt \right|
\]
\begin{equation}
\label{lmdeq1}
\leq e^{H} H^4 \frac{r \alpha_k}{|k_2|} \leq e^H H^4 \frac{1}{k_2^2 \alpha_k |f''(s_k)|}
\end{equation}
and
\[
\left|  \left( \int_{t_j}^{s_k-\alpha_k} +\int_{s_k+\alpha_k}^{t_{j+1}} \right)
(e^{-ir(t k_1 + f(t) k_2)}-e^{-ir(s_k k_1 + f(s_k) k_2)})\frac{f''(t)}{(k_1 + k_2 f'(t))^2}dt \right|
\]
\[
\leq \left| \frac{2}{k_2(k_1+k_2 f'(t))} \left( \left.  \right|_{t_j}^{s_k-\alpha_k} +
 \left. \right|_{s_k+\alpha_k}^{t_{j+1}} \right) \right|
\]
\begin{equation}
\label{lmdeq2}
\leq 2 \left| \frac{f'(t_{j})-f'(t_{j+1})}{(k_1+k_2 f'(t_{j}))(k_1+k_2 f'(t_{j+1}))} \right| + 4H\frac{1}{k_2^2 \alpha_k |f''(s_k)|}
\end{equation}
Lemmas \ref{lemma2} and \ref{lemma1} complete the proof of Theorem \ref{lmc}.
\end{proof}

\begin{lemma}
\label{lemma2}
We have the following bound on the sum over $k \in E$ satisfying $r \geq B(k_1,k_2)$,
\[
\sum_{r \geq B(k_1,k_2)} \frac{r}{(1+r^{-1/3}|k|)^\nu}
\left| \int_{t_j}^{t_{j+1}}
\frac{e^{-ir(t k_1 + f(t) k_2)}}{(k_1^2 + k_2^2)}(k_2 - k_1 f^{'}(t)) dt \right|
\]
\[
= O(r^{2/3})
\]
\end{lemma}
\begin{proof}
We know from \eqref{lmeeq1} and \eqref{lmeeq2} that we need only to bound
\[
\sum_{r \geq \alpha_k^{-2} |k_2 f''(s_k)|^{-1}} \frac{r^{1/2}}{(1+r^{-1/3}|k|)^\nu |k_2^3 f''(s_k)|^{1/2}}
\]
since the calculations made for the case $k \notin E$ can be applied to the sum over $\left| \frac{(f'(t_{j})-f'(t_{j+1}))}{(k_1+k_2 f'(t_{j}))(k_1+k_2 f'(t_{j+1}))} \right|$.
For a given compact interval $[\beta_1,\beta_2]$ with $\beta_1 > 0$, the contribution of the pairs $(k_1,k_2)$ such that $k_1+k_2 f'(s_k) = 0$ and
$|f''(s_k)| \in [\beta_1,\beta_2]$ will be bounded, for some $C > 0$, by
\[
C r^{1/2} \sum_{k_2=1}^{+\infty} \frac{1}{|k_2|^{3/2}} \sum_{k_1=0}^{ C k_2 } \frac{1}{(1+r^{-1/3}|k|)^\nu} = O(r^{2/3})
\]

Now suppose $f''(t)$ vanishes at $t_j$.
Since the curvature admits only zeros of first order, we will obtain for some $\epsilon$
\[
|f''(t)| \geq \epsilon |f'(t)-f'(t_j)|^{1/2}
\]
if $0 \leq t-t_j \leq \epsilon$.
Since $f'(t_j)$ is typical, we also have
\[
\left| f'(s_k) - f'(t_j) \right| =
\left| \frac{k_1}{k_2} + f'(t_j) \right| \geq \frac{\delta}{k_2^{2} \log(1+|k_2|)^{\tau}}
\]
The contribution of the terms such that $s_k - t_j \in [0,\epsilon]$ will be
bounded by
\[
r^{1/2} \sum_{k_2 = 1}^{+\infty} \frac{C}{|k_2|^{3/2}(1+r^{-1/3}k_2)^\nu} \sum_{n=0}^{\lfloor k_2 |f'(t_j+\epsilon) - f'(t_j)| \rfloor}
\left( \frac{\delta}{k_2^{2} \log(k_2 +1)^{\tau}}+\frac{n}{k_2} \right)^{-1/4}
\]
\[
\leq \frac{C r^{1/2}}{\delta^{1/4}} \sum_{k_2 = 1}^{+\infty} \frac{\log(1+k_2)^{\tau/4}}{k_2(1+r^{-1/3}k_2)^\nu}
\]
\[
+C r^{1/2} \sum_{k_2 = 1}^{+\infty} \frac{1}{|k_2|^{5/4}(1+r^{-1/3}k_2)^\nu}
\int_{0}^{k_2|f'(t_j+\epsilon) - f'(t_j)|} \frac{du}{u^{1/4}}
\]
\[
= O(r^{2/3})
\]
Note that
\[
\sum_{k_2 = 1}^{+\infty} \frac{\log(1+k_2)^{\tau/4}}{k_2(1+r^{-1/3}k_2)^\nu}
\leq C \int_1^{+\infty} \frac{\log (1+u)^{\tau/4} du}{u(1+r^{-1/3}u)^\nu}
\]
and
\[
\int_1^{r^{1/3}} \frac{\log (1+u)^{\tau/4} du}{u(1+r^{-1/3}u)^\nu} = O(\log(r)^{1+\tau/4})
\]
\[
\int_{r^{1/3}}^{+\infty} \frac{\log (1+u)^{\tau/4} du}{u(1+r^{-1/3}u)^\nu} = O(\log(r)^{\tau/4})
\]
We also have
\[
\sum_{k_2 = 1}^{+\infty} \frac{1}{|k_2|^{5/4}(1+r^{-1/3}k_2)^\nu}
\int_{0}^{k_2|f'(t_j+\epsilon) - f'(t_j)|} \frac{du}{u^{1/4}}
\]
\[
\leq C \sum_{n = 1}^{+\infty} \frac{1}{n^{1/2}(1+r^{-1/3}n)^\nu} = O(r^{1/6})
\]

The function $f''(t)$ diverges at $t=F_1(c_2)$ and $t=F_1(c_3)$.
We will study the case of $F_1(c_3)$, as the other one is equivalent after swapping the axes.
Suppose that the singular point is at $t_j$. When $t-t_j$ is small enough,
\[
M \log (t-t_j) < f'(t) < m \log (t-t_j)
\]
and
\[
m (t-t_j)^{-1}     < |f''(t)| < M (t-t_j)^{-1}
\]
for some $0 < m < M$.
We deduce
\[
|f''(t)| > m \exp(-M^{-1}f'(t))
\]
The contribution of the terms for which $s_k$ is near $F_1(c_3)$ will then be bounded by
\[
r^{1/2}\sum_{k_2=1}^{\infty} \frac{1}{|k_2|^{3/2}} \sum_{k_1=0}^{+\infty}
\frac{\exp \left( -k_1 (2M k_2)^{-1} \right)}
{m^{1/2} (1+r^{-1/3}|k|)^\nu}
\]
\[
\leq \frac{C r^{1/2}}{m^{1/2}} \sum_{k_2=1}^{\infty} \frac{(1+2Mk_2)}{|k_2|^{3/2}(1+r^{-1/3}k_2)^\nu}
\]
\[
= O(r^{2/3})
\]
Similar bounds hold if we consider $t_{j+1}$, so the sum is of order $O(r^{2/3})$.
\end{proof}

\begin{lemma}
\label{lemma1}
We have the following bound on the sum over $k \in E$ satisfying $r < B(k_1,k_2)$,
\[
\sum_{r < B(k_1,k_2)} \frac{r}{(1+r^{-1/3}|k|)^\nu}
\left| \int_{t_j}^{t_{j+1}}
\frac{e^{-ir(t k_1 + f(t) k_2)}}{(k_1^2 + k_2^2)}(k_2 - k_1 f^{'}(t)) dt \right|
\]
\[
= O(r^{2/3})
\]
\end{lemma}
\begin{proof}
We use \eqref{lmdeq1} and \eqref{lmdeq2} so that it suffices to bound
\[
\sum_{\alpha_k< |r k_2 f''(s_k)|^{-1/2}} \frac{1}{(1+r^{-1/3}|k|)^\nu k_2^2 \alpha_k |f''(s_k)|}
\]
If $m < |f''(t)| < M$ near $t_j$ and $f'(t_j)$ is typical, then
\[
2 M \alpha_k > \left| f'(s_k) - f'(t_j) \right| = \left| \frac{k_1}{k_2} + f'(t_j) \right| \geq \frac{\delta}{k_2^2 \log(1+|k_2|)^{\tau}}
\]
Suppose $f''(t)$ vanishes at $t_j$.
Since the zeros of $f''(t)$ are of first order, $m (t-t_j) < |f''(t)| < M(t-t_j)$ near this point.
Then $|f''(s_k)|> 2m \alpha_k$ and the fact that $f'(t)$ is typical at a zero of $f''(t)$ also implies
\[
2 M \alpha_k^2 > \left| f'(s_k) - f'(t_j) \right| = \left| \frac{k_1}{k_2} + f'(t_j) \right| \geq \frac{\delta}{k_2^2 \log(1+|k_2|)^{\tau}}
\]
If $t_j=F_1(c_3)$, $f''(t)$ diverges like $(t-t_j)^{-1}$ so $\alpha_k |f''(s_k)|$ is bounded
from below by a strictly positive constant when $s_k$ approaches $t_j$.
Similar bounds hold if $s_k$ is near $F_1(c_2)$. We deduce
\[
\sum_{\alpha_k< |r k_2 f''(s_k)|^{-1/2}} \frac{1}{(1+r^{-1/3}|k|)^\nu k_2^2 \alpha_k |f''(s_k)|}
\]
\[
\leq \sum_{k_2 = 1}^{+ \infty} \frac{C \log(1 + k_2)^{\tau}}{(1+ r^{-1/3}k_2)^\nu}  + \sum_{k \neq 0} \frac{C}{(1+r^{-1/3}|k|)^\nu}
\]
\[
= O(r^{2/3})
\]
\end{proof}

\begin{lemma}
\label{lmd}
There exists $H>0$ such that the functions
$$
g_{k}(t) = \frac{1}{f''(s_k)(t-s_k)}\int_{s_k}^t f''(u)du
$$
and
$$
h_k(t) = 2\frac{\int_{s_k}^t (t-u)f''(u)du}{f''(s_k)(t-s_k)^2}
$$
admit the bounds $0 < 1/H < g_k, h_k < H$ in the interval $t \in [s_k-\alpha_k,s_k +\alpha_k]$, for $k \in E$.
We recall that $k_1+k_2 f'(s_k) = 0$ and $\alpha_k= \frac{1}{2}\min \{ s_k - t_j, t_{j+1}-s_k \}$.
\end{lemma}
\begin{proof}[Proof of Lemma $\ref{lmd}$]
The result is clear if we fix a compact interval $[T_j,T_{j+1}] \subset (t_j,t_{j+1})$ and consider the case of $k$ such
that $s_k \in [T_j,T_{j+1}]$. We are left to study the ends of the interval $(t_j,t_{j+1})$. Suppose that $f''(t)$ vanishes at $t_j$, we have $c(t-t_j) < |f''(t)| < C(t-t_j)$ for $(t-t_j)$ sufficiently small and, assuming $t \in [s_k-\alpha_k,s_k +\alpha_k]$,
$$
\frac{c}{2C} < g_k(t) < \frac{3C}{2c}
$$
with the same bounds for $h_k(t)$. Now suppose that $t_j = F_1(c_3)$, we know that in this case we have
\[
c (t-t_j)^{-1} < |f''(t)| < C(t-t_j)^{-1}
\]
and we deduce that
$$
\frac{2c}{3C} < g_k(t) < \frac{2C}{c}
$$
with the same bounds for $h_k(t)$.
The same arguments can be applied for $t_j=F_1(c_2)$ or at $t_{j+1}$.
This shows the existence of uniform bounds on $g_k,h_k$ for $k \in E$.
\end{proof}

\subsection{Lattice points near straight lines}
\label{straight}
Our previous calculations dealt with integrals taken over the curve $\gamma$.
The boundaries of the domains considered in the proof of Theorem $\ref{thm1}$ also contain
straight parts.
We will need the following results to complete our study of $\tilde{N}_a(r,D)$.
Note that $\Psi$ is the cut-off function introduced at the beginning of section \ref{section31}.

\begin{theorem}
\label{lmb}
If $\omega$ is a line segment emanating from the origin, that has finite length, with rational or typical slope, then
\[
\sum_{k \neq 0} r
\cos(\langle a,k \rangle) \hat{\Psi}(r^{-1/3}k)
\int_{\omega} \frac{e^{-ir(x k_1 + y k_2)}}{(k_1^2 + k_2^2)}(k_2 dx - k_1 dy)
=O(r^{2/3})
\]
\end{theorem}
\begin{proof}
We first suppose that $\omega$ is rational and given by $(pt,qt)$ for $t \in [0,\ell]$ and some $p,q \in \mathbb Z$.
The summation over $k \neq 0$ is divided in two parts.
The first part contains the terms such that $p k_1 + q k_2 = 0$.
Since $\hat{\Psi}(r^{-1/3}k)$ depends only on $r^{-1/3}|k|$ and the terms of the sum are antisymmetric in $k$,
\[
\sum_{\substack{p k_1 + q k_2 = 0\\ k \neq 0}}
r \cos(\langle a,k \rangle) \hat{\Psi}(r^{-1/3}k)
\int_0^\ell \frac{e^{-ir(p k_1 + q k_2)t}}{(k_1^2 + k_2^2)} (p k_2 - q k_1) dt
\]
\[
= \sum_{\substack{p k_1 + q k_2 = 0\\ k \neq 0}}
r \ell \cos(\langle a,k \rangle) \hat{\Psi}(r^{-1/3}k) \frac{(p k_2 - q k_1)}{(k_1^2 + k_2^2)} = 0
\]
The other part contains $(k_1,k_2)$ such that $|p k_1 + q k_2| \geq 1$, and is bounded in absolute value by
\[
\sum_{p k_1 + q k_2 \neq 0} \frac{r}{(1+r^{-1/3}|k|)^\nu}
\left| \int_0^\ell \frac{e^{-ir(p k_1 + q k_2)t}}{(k_1^2 + k_2^2)} (p k_2 - q k_1) dt \right|
\]
\[
\leq  \sum_{p k_1 + q k_2 \neq 0} \frac{2}{(1+r^{-1/3}|k|)^\nu|k|^2}
\left| \frac{p k_2 - q k_1}{p k_1 + q k_2} \right|
\]
\[
\leq C \int_1^{+ \infty} \frac{du}{(1+r^{-1/3}u)^\nu} = O(r^{1/3})
\]

If $\omega$ is represented by $(t, \alpha t)$, with $t \in [0,\ell]$ and $\alpha$ a typical number,
we have
\[
\sum_{| k_1 + \alpha k_2| \geq 1} \frac{r}{(1+r^{-1/3}|k|)^\nu}
\left| \int_0^\ell \frac{e^{-ir( k_1 + \alpha k_2)t}}{(k_1^2 + k_2^2)} ( k_2 - \alpha k_1) dt \right|
\]
\[
\leq  \sum_{|k_1 + \alpha k_2| \geq 1} \frac{2}{(1+r^{-1/3}|k|)^\nu}
\frac{|k_2 - \alpha k_1|}{(k_1^2 + k_2^2)}
= O(r^{1/3})
\]
Since $\alpha$ is typical, we also have the following upper bound,
\[
\sum_{\substack{| k_1 + \alpha k_2| < 1\\ k \neq 0}} \frac{r}{(1+r^{-1/3}|k|)^\nu}
\left|
\int_0^\ell \frac{e^{-ir( k_1 + \alpha k_2)t}}{(k_1^2 + k_2^2)} ( k_2 - \alpha k_1) dt \right|
\]
\[
\leq
\sum_{\substack{| k_1 + \alpha k_2| < 1\\ k \neq 0}}
\frac{2}{(1+r^{-1/3}|k|)^\nu|k|^2}
\left| \frac{ k_2 - \alpha k_1}{ k_1 + \alpha k_2} \right|
\]
\[
\leq \sum_{n=1}^{+\infty}
\frac{C \log(1+n)^{\tau}}{(1+r^{-1/3}n)^\nu }  = O(r^{1/3}\log(r)^{\tau})
\]
for $C$ sufficiently large.
\end{proof}

\begin{lemma}
\label{lme}
If $\omega$ is a ray of length $r>1$, with rational or typical slope, then
\[
\sharp \{ v \in \Gamma_a | 0 < {\rm{dist}}(v,\omega) \leq r^{-1/3} \} \leq C r^{2/3}
\]
for some $C > 0$ which depends on $a$ and the direction of $\omega$.
\end{lemma}
\begin{proof}
If the slope is rational, $\omega$ is contained in $\bar{\omega}=\{(pt,qt)| t \in \mathbb{R}\}$ for some
$(p,q) \in \mathbb{Z}^2$.
In this case, there is a minimal distance between $\omega$ and the points
of the lattice $\Gamma_a$ not in $\bar{\omega}$.
Thus the existence of the bound is clear.

Suppose the slope is typical and that $\omega$ is represented by $(t \cos \theta,t \sin \theta)$ with $\tan \theta$ a typical number
and $t \in [0, r]$.
Let $B$ be the set of points in $\mathbb{R}^2$ which are at a distance of less than $2 r^{-1/3}$ to $\omega$,
and $\chi_B$ the characteristic function of $B$. It suffices to bound
\[
\sum_{k \in 2 \pi \mathbb Z^2 +a} (\Psi_{r^{-1/3}} \ast \chi_{B}) (k)
\]
which, by the Poisson summation formula, is equal to
\[
\frac{1}{4 \pi^2} \sum_{k \in \mathbb Z^2} \exp(i\langle a,k \rangle) \hat{\Psi}(r^{-1/3}k) \hat{\chi}_{B}(k)
\]
The term corresponding to $k=0$ is the area of $B$, which is $4 r^{2/3}(1+\pi r^{-4/3})$.
If $k \neq 0$,
\[
i\hat{\chi}_{B}(k) = \oint_{\partial B} \frac{e^{-i(x k_1+y k_2)}}{k_1^2+k_2^2} (k_2 dx - k_1 dy)
\]
\[
=
2 \sin \left(2 r^{-1/3} (\cos \theta k_2 - \sin \theta k_1) \right)
\int_{0}^{r} \frac{e^{-i t ( \cos \theta k_1+ \sin \theta  k_2)}}{k_1^2+k_2^2} ( \cos \theta k_2- \sin \theta k_1) dt
\]
\[
+ 2 r^{-1/3}\int_\theta^{\theta + \pi} \frac{e^{-i(-k_1 \sin s + k_2 \cos s)}}{k_1^2+k_2^2} (- k_2 \cos s  +  k_1 \sin s) ds
\]
\[
+ 2 r^{-1/3} e^{-ir (k_1 \cos \theta + k_2 \sin \theta)} \int_\theta^{\theta + \pi}
\frac{e^{-i(k_1 \sin s - k_2 \cos s)}}{k_1^2+k_2^2} (k_2 \cos s - k_1 \sin s) ds
\]
The last two terms are of order $r^{-1/3}|k|^{-1}$, so their contribution in the sum is bounded by a constant
\[
\sum_{k \neq 0} \frac{r^{-1/3}}{|k|(1+r^{-1/3}|k|)^\nu} < C
\]
We must then bound the integrals on the parts of $\partial B$ parallel to $\omega$.
But
\[
\left| \int_{0}^{r} \frac{e^{-i t ( \cos \theta k_1+ \sin \theta  k_2)}}{k_1^2+k_2^2} ( \cos \theta k_2- \sin \theta k_1) dt \right|
\leq \frac{2}{k_1^2+k_2^2} \frac{|\cos \theta k_2- \sin \theta k_1|}{|\cos \theta k_1+ \sin \theta  k_2|}
\]
and, since $\tan \theta$ is typical, their contribution is bounded by
\[
\sum_{k \neq 0} \frac{1}{|k|(1+r^{-1/3}|k|)^\nu} + \sum_{n=1}^{+\infty} \frac{\log(1+n)^{\tau}}{(1+r^{-1/3}n)^\nu}
=O(r^{1/3}\log(r)^{\tau})
\]
\end{proof}

The following lemma estimates the variation of the total number of points from $\mathbb Z^2 + (0,\beta)$ and $\mathbb Z^2 - (0,\beta)$, contained in the region $\{(x,y)|0 < x \leq r,0 \leq y \leq \alpha x \}$ of the plane, in function of $\beta$.  The points lying on the $x$ axis are given a weight of $\frac{1}{2}$.

\begin{lemma}
\label{huxley}
Let $\alpha$ be a typical number, and define the following function
\[
K(\alpha,r,\beta)=
\begin{cases}
r +  2 \sum_{1 \leq k \leq r} \lfloor \alpha k \rfloor          & \text{ if } \beta \in \mathbb{Z} \\
\sum_{1 \leq k \leq r}  \left( \lfloor \alpha k + 1 - \beta \rfloor  + \lfloor \alpha k +\beta \rfloor \right)  & \text{ if } \beta \notin \mathbb{Z}\\
\end{cases}
\]
for $\beta \in \mathbb{R}$.
Then for some $C_{\alpha},\tau > 0$ and $r$ sufficiently large we have
\[
\left| K(\alpha,r,\beta_1) - K(\alpha,r,\beta_2) \right| < C_{\alpha} \log(r)^{1+\tau}
\]
for any $\beta_1,\beta_2 \in \mathbb{R}$.
\end{lemma}
\begin{proof}
We consider the function $\sigma(t) = \lfloor t \rfloor - t + \frac{1}{2}$, so that
it is sufficient to show
\[
\left| \sum_{1 \leq k \leq r} \sigma(\alpha k + \beta) \right| < C_{\alpha} \log(r)^{1+\tau}
\]
for all $\beta \in \mathbb{R}$.

Suppose $\alpha \in \mathbb R$
is irrational,
then we can write $\alpha$ in a unique way as a continued fraction
\[
\alpha = a_0 + \cfrac{1}{a_1+\cfrac{1}{a_2+\cfrac{1}{a_3+\cfrac{1}{\dots}}}}
\]
with $a_i \in \mathbb Z$ and $a_i > 0$ for $i >0$.
Also, any sequence $(a_0;a_1,a_2,a_3,\dots)$
respecting the previous conditions will represent an irrational number.
The numbers $\{a_n\}$ are called the partial quotients of $\alpha$.
We define the convergents of $\alpha$ as
\[
\frac{P_n}{Q_n} = a_0 + \cfrac{1}{a_1+\cfrac{1}{a_2+\cfrac{1}{\dots+\cfrac{1}{a_n}}}}
\]
The denominators $Q_n$ can also be defined by a recurrence relation, $Q_0=1$, $Q_1=a_1$ and
$Q_{n+1}= a_{n+1}Q_n + Q_{n-1}$ for $n>0$.

Supposing $r$ is an integer,
let $m$ be the largest integer such that $Q_m \leq r$
and consider
\[
b_m = \lfloor r / Q_m \rfloor, \qquad r = b_m Q_m + N_{m-1}
\]
with
\begin{equation}
\label{discrep1}
b_j = \lfloor N_j / Q_j \rfloor, \qquad N_j = b_j Q_j + N_{j-1}
\end{equation}
for $0 \leq j \leq m-1$. Note that $N_{j-1} < Q_j$ and $Q_0=1$, so $N_{-1}=0$ and
\[
r = \sum_{j=0}^{m} b_j Q_j
\]

Lemma 2.3.2 of \cite{8} shows
\[
\left| \sum_{1 \leq k \leq r} \sigma(\alpha k +b) \right| \leq
1+ m + \frac{5}{2} \sum_{j=0}^{m} b_j
\]
so it suffices to find appropriate bounds on $m$ and $b_j$ using the fact that $\alpha$ is typical.
We also know from Lemma 1.5.2 of \cite{8} that
\[
\left| \alpha - \frac{P_j}{Q_j} \right| < \frac{1}{Q_{j} Q_{j+1}} \leq \frac{1}{2^j}
\]
but since for any $q > 0$
\[
\left| \alpha - \frac{p}{q} \right| \geq \frac{\delta}{q^2 \log(1+q)^{\tau}}
\]
we deduce
\[
\delta 2^j < Q_j^2 \log(1+Q_j)^{\tau}
\]
and that for any fixed $\epsilon > 0$, and $C_{\epsilon} > 0$ large enough
\begin{equation}
\label{discrep2}
\left( \delta 2^j \right)^{1/(2+\epsilon)} < Q_j < C_{\epsilon} Q_{j-1} \log(1+Q_{j-1})^{\tau}
\end{equation}
\begin{equation}
m \leq \frac{1}{\log 2} \log \left( \frac{r^{2+\epsilon}}{\delta} \right)
\end{equation}
Thus we are left to show that the sum $\sum_j b_j$ is of order $\log(r)^{1+\tau}$.
Using \eqref{discrep1} and \eqref{discrep2} we deduce
\[
\sum_j b_j \leq \frac{r}{Q_m} + \sum_{j=0}^{m-1} \frac{Q_{j+1}}{Q_j}
\]
\[
\leq C_{\epsilon} \sum_{j=0}^{m} \log(1+Q_{j})^{\tau}
\]
\[
\leq C \log(r) \log(1+r)^{\tau}
\]
\end{proof}

\subsection{Bound on the remainder term}
\label{boundremterm}
Our results on lattice counting can be applied to the
eigenvalue counting, using the correspondence established in section \ref{section23}, with a
precision of order $O(r^{2/3})$.
Note that the Riemannian volume of $T$ and the Euclidean area of $A$ are related in the following way
\[
\frac{1}{4 \pi} \Area(T) = \frac{1}{(2\pi)^2} \int_{H(p,q)\leq 1} dp dq = \frac{1}{\pi^2} \Area(A)
\]

This is justified by the fact that for $c_3 < c < c_2$ the geodesic flow occur on $4$ distinct tori in phase space,
corresponding to the possible signs of $\dot{q}_1,\dot{q}_2$, on which $I_1,I_2$ take the same value.
For $c_4 < c < c_3$ there are $2$ tori, corresponding to the sign of $\dot{q}_1$, however the flow induced
by $I_2$ has period $2$ instead of $1$. Thus the integration in the angle variables is doubled in this region.
The case of $c_2 < c < c_1$ is similar if we consider $\dot{q}_2$ and $I_1$.
The symplectic volume of the set satisfying $H(p,q) \leq 1$ in $\textbf{T}^{*}(T)$ is then $4\, \Area(A)$.

\begin{proof}[Proof of Theorem $\ref{thm1}$]
Given $r$ large enough, we can bound the difference between $R(r)$ and $2 N_{0}(A, r)$ using equation \eqref{eq1}.
We know that to each pair $(m_1,m_2)$, with $m_1,m_2 \geq 0$, correspond a pair $(\lambda,c)$ which satisfies
\[
\Phi_0(\lambda,c) + \Phi_1(\lambda,c) + \Phi_2(\lambda,c) = 2 \pi \left( \left[ \frac{m_1+1}{2} \right],\left[ \frac{m_2+1}{2}\right] \right)
\]
Since $\Phi_0(\lambda,c)=\lambda(F_1(c), F_2(c))$, our first approximation is to count the number of pairs $(m_1,m_2)$
such that
\[
2 \pi \left( \left[ \frac{m_1+1}{2} \right],\left[ \frac{m_2+1}{2}\right] \right) \in r A
= r (A_1 \cup A_2 \cup A_3)
\]
which is given by $2N_0(A,r)$.
Theorems \ref{lma}, \ref{lmc} and \ref{lmb} applied to the domain $A$ show that
\[
\tilde{N}_{0}(A, r)- \frac{r^2}{2\pi^2} \Area(A)  =O(r^{2/3})
\]
By \eqref{varib}, and since we give a weight of $\frac{1}{2}$ to the points of $\Gamma_0$ lying on the coordinate axes, we have
\begin{equation}
\label{sandwich}
\tilde{N}_{0}(A, r- \varkappa r^{-1/3})
\leq N_{0}(A, r)
\leq \tilde{N}_{0}(A, r+ \varkappa r^{-1/3})
\end{equation}
for $\varkappa$ sufficiently large, so that
\[
N_{0}(A, r) - \frac{ r^2}{2 \pi^2}\Area(A) =O(r^{2/3})
\]
We are then left to show that the corrections, brought by $\Phi_1$ and $\Phi_2$, to this approximation
generate an error term bounded by $O(r^{2/3})$.

The first discrepancies considered are the points $2 \pi \left( \left[ \frac{m_1+1}{2} \right],\left[ \frac{m_2+1}{2}\right] \right)$ for which
the solution of \eqref{eq1} satisfies $|\lambda(m_1,m_2)-r| \leq \text{Const}$, with $|c(m_1,m_2) - c_2| \leq \text{const} \lambda^{-2/3}$ or $|c(m_1,m_2) - c_3| \leq \text{const} \lambda^{-2/3}$.
In such cases, $\Phi_1$ and $\Phi_2$ are bounded by a constant independent of $(\lambda,c)$.
The maximal number of such points is of order $r^{1/3} \log r$. Indeed, an interval of order $r^{-2/3}$ in the $c$ variable,
around $c_2$ or $c_3$, translates into an angular
interval of order $r^{-2/3} \log r$. This is verified using the asymptotics $(3)$ in Theorem $\ref{asympt1}$, (also Lemma 6.4 in \cite{1}).

We then consider the points in $r A_1$.
In this case, the function $\Phi_1$ will
induce a transformation from the lattice $\Gamma_0$ to $\Gamma_{(0,\frac{\pi}{2})}$ and $\Gamma_{(0,-\frac{\pi}{2})}$.
Only points near the boundary $\gamma$, in $(r + \varkappa) A_1 \setminus (r-\varkappa)A_1$ for some $\varkappa$ sufficiently large,
might be affected by those corrections.
Using Lemma \ref{lme}, we can apply an inequality similar to \eqref{sandwich} for $A_1$. Thus
\[
N_{a}(A_1, r) - \frac{ r^2}{2 \pi^2}\Area(A_1) =O(r^{2/3})
\]
for any $a \in \mathbb{R}^2$. However, we must only count the points leaving or entering $rA_1$ through $r\gamma$
during the transformation of the lattice.
Lemma \ref{huxley} shows that $\log(r)^{1+\tau}$ points pass through the line $(t F_1(c_3), tF_2(c_3))$, for $t \in \mathbb R$.
This means that adding $\Phi_1$ (which is equivalent to shifting the lattice in the region $r A_1$ only)
yields a correction of order at most $O(r^{2/3})$.
Since $|\Phi_2| < \text{ Const }  r^{-2/3} \log r$ around $r (\gamma \cap A_1)$, $\Phi_2$ also changes the count
by $O(r^{2/3})$ only.


A similar argument holds in $r A_3$, where we consider $\Gamma_{(\frac{\pi}{2},0)}$ and $\Gamma_{(-\frac{\pi}{2},0)}$,
and $r A_2$, where we do not have to change the lattice.
We deduce that, for large enough $r$, the errors occurring in our first approximation $2N_0(A,r)$
are all contained in a $O(r^{2/3})$ term and that
\[
R(r)- \frac{r^{2}}{4 \pi} \Area(T) = O(r^{2/3})
\]



\end{proof}

\subsection{Density of the nondegenerate metrics}
\label{density}

Consider the set $\Omega$ of pairs of functions $(U_1,U_2)$ satisfying the conditions described in Definition \ref{condition11}.
We put on $\Omega$ the topology induced by the Whitney topology of $C^{\infty}(\mathbb R / \mathbb Z) \times C^{\infty}(\mathbb R / \mathbb Z)$.
Note that $\Omega$ will be an open subset of $C^{\infty}(\mathbb R / \mathbb Z) \times C^{\infty}(\mathbb R / \mathbb Z)$.

\begin{proof}[Proof of Theorem $\ref{thm2}$]
We first show that the set of metrics satisfying the condition $(1)$ of Definition \ref{genericity} is open and dense.
It is open since the curvature $\kappa(c)$ and its derivatives are continuous functions of $\Omega$.
Any pair $(U_1,U_2)$ can be approximated in $\Omega$ by analytic functions, for example using their partial Fourier series.
The curvature $\kappa(c)$ of $\gamma$ corresponding to a pair of analytic functions is also analytic.
Since the curvature diverges at $c_2$ and $c_3$, its zeros are located in a compact subset of
$[c_4,c_3) \cup (c_2,c_1]$.
We deduce that $\kappa(c)$ has  a finite number of zeros, each of finite order. If some zeros are of order greater than one, we must modify slightly $(U_1,U_2)$ to
make them first order.
Remember that the curvature vanishes if and only if $(F_2^{''}F_1^{'} - F_1^{''}F_2^{'})(c)$ vanishes,
and that their zeros are of the same order.
Suppose $\kappa(\tilde{c})=0$ and the order of vanishing is greater than one.
It is sufficient to find $\tilde{U}=(\tilde{U}_1,\tilde{U}_2)$ such that
\[
\frac{d}{d \epsilon} \left. \kappa_{U+ \epsilon \tilde{U}}(\tilde{c}) \right|_{\epsilon=0}  \neq 0
\]
If $\tilde{c} \in [c_4,c_3)$, we might take $\tilde{U}_2 = 0$
and
\[
\tilde{U}_1 (q_1) =
\frac{F_2^{''}(\tilde{c})}{4}   \frac{1}{(U_1(q_1)-\tilde{c})^{3/2}} -
\frac{3 F_2^{'}(\tilde{c})}{8}  \frac{1}{(U_1(q_1)-\tilde{c})^{5/2}}
\]
so that
\[
\frac{F_2^{''}(\tilde{c})}{4}   \int_{0}^{1} \frac{\tilde{U}_1(q_1) dq_1}{(U_1(q_1)-\tilde{c})^{3/2}} -
\frac{3 F_2^{'}(\tilde{c})}{8}   \int_{0}^{1} \frac{\tilde{U}_1(q_1) dq_1}{(U_1(q_1)-\tilde{c})^{5/2}}
\neq 0
\]
If $\tilde{c} \in (c_2,c_1]$, we might take $\tilde{U}_1 = 0$
and
\[
\tilde{U}_2 (q_2)=
-\frac{3 F_1^{'}(\tilde{c})}{8}       \frac{1}{(\tilde{c} - U_2(q_2))^{5/2}} -
\frac{F_1^{''}(\tilde{c})}{4}        \frac{1}{(\tilde{c} - U_2(q_2))^{3/2}}
\]
so that
\[
-\frac{3 F_1^{'}(\tilde{c})}{8}       \int_{0}^{1} \frac{\tilde{U}_2(q_2) dq_2}{(\tilde{c} - U_2(q_2))^{5/2}} -
\frac{F_1^{''}(\tilde{c})}{4}       \int_{0}^{1} \frac{\tilde{U}_2(q_2) dq_2}{(\tilde{c} - U_2(q_2))^{3/2}}
\neq 0
\]
By taking $\epsilon$ small enough, the zeros of $\kappa_{U+ \epsilon \tilde{U}}$ in a neighbourhood
of $\tilde{c}$ will be of first order, where $U + \epsilon \tilde{U} = (U_1 + \epsilon \tilde{U}_1 , U_2 + \epsilon \tilde{U}_2)$.
We deduce that the set of metrics for which $(1)$ holds in Definition \ref{genericity} is also dense.

The last part of the proof requires to find a perturbation $\tilde{U}$ of $(U_1,U_2)$ so that
the derivatives in $\epsilon$ of the functions from conditions $(2)$,$(3)$ and $(4)$, in
Definition \ref{genericity}, do not vanish.
Assuming
$(1)$ is satisfied, we can keep track of each zero $\tilde{c}=\tilde{c}(\epsilon)$ of $\kappa_{U+ \epsilon \tilde{U}}$.
However, their variation does not influence the derivatives in $\epsilon$ of ${F_{2}}'(\tilde{c})/{F_{1}}'(\tilde{c})$.
We might suppose that $c_1$ and $c_4$ are not zeros of $\kappa(c)$, and that $\tilde{U_i}$ vanishes around the critical points of $U_i$, so
the value of $c_j$ remains constant, $j=1,2,3,4$.
We assume additionally that
the image of $ \text{supp}(\tilde{U_1})$ under $U_1$ is close
enough to $c_1$ and does not contain any zero of $\kappa$. In the same way, the image of $ \text{supp}(\tilde{U_2})$ under $U_2$ must be near $c_4$ and not contain any zero of $\kappa$.
In this setting, we require
\[
F_1^{'}(c)
\int_{U_2(q_2) \leq c} \frac{\tilde{U}_2(q_2) dq_2}{(c-U_2(q_2))^{3/2}}
-
F_2^{'}(c)
\int_{U_1(q_1) \geq c} \frac{\tilde{U}_1(q_1) dq_1}{(U_1(q_1)-c)^{3/2}}
\neq 0
\]
at each zero for condition $(2)$.
Also
\[
\int_0^1 \frac{\tilde{U}_2(q_2) dq_2}{(c_1-U_2(q_2))^{3/2}} \neq 0 \qquad \text{ and } \qquad
\int_0^1 \frac{\tilde{U}_1(q_1) dq_1}{(U_1(q_1)-c_4)^{3/2}}\neq 0
\]
for condition $(3)$,
\[
\int_0^1 \frac{\tilde{U}_2(q_2) dq_2}{(c_j-U_2(q_2))^{1/2}} \int_0^1 (U_1(q_1)-c_j)^{1/2} dq_1
\]
\[
+
\int_0^1 (c_j-U_2(q_2))^{1/2} dq_2 \int_0^1 \frac{\tilde{U}_1(q_1) dq_1}{(U_1(q_1)-c_j)^{1/2}}
\neq 0
\]
with $j=2,3$ for condition $(4)$. We can obviously find such a pair $(\tilde{U}_1,\tilde{U}_2)$.

Then, for any zero $\tilde{c}$ of $\kappa$, the function $\left({F_{2}}'(\tilde{c})/{F_{1}}'(\tilde{c})\right)(\epsilon)$ is a diffeomorphism between $(-\delta,\delta) \ni \epsilon$
and an interval in $\mathbb{R}$.
Since $\kappa$ has a finite number of zeros and almost all real numbers are typical, the set of $\epsilon$ for which $(U_1+\epsilon \tilde{U}_1,U_2+\epsilon \tilde{U}_2)$ meets condition $(2)$ is also
of full measure. The same argument can be applied in conditions $(3)$ and $(4)$.
We deduce that there is an $\epsilon$, as small as required, for which the metric given by $(U_1+\epsilon \tilde{U}_1,U_2+\epsilon \tilde{U}_2)$ is nondegenerate.
\end{proof}

\subsection{Limitations of the method}
\label{limits}
We cannot show that the $O(\lambda^{2/3})$ bound holds for a set of second Baire category of metrics, using these methods.
This would require that the following bound, assumed throughout section \ref{section31}, holds for $\alpha$ in a subset of second Baire category in $\mathbb R$,
\[
\sum_{k=1}^{+\infty} \frac{1}{(1+r^{-1/3}k)^\nu} \frac{1}{k ||\alpha k||} = O(r^{2/3}) \text{ for some } \nu > 0
\]
where $|| \alpha ||$ is the distance to the nearest integer from $\alpha$.
This is impossible since we can construct a denumerable intersection of open dense subsets of $\mathbb R$
in which it does not hold.
Note that since
\[
\sum_{k=1}^{+\infty} \frac{1}{(1+r^{-1/3}k)^\nu} \frac{1}{k ||\alpha k||} \geq
\frac{1}{2^\nu} \sum_{k=1}^{r^{1/3}} \frac{1}{k ||\alpha k||}
\]
it is sufficient to show
\begin{equation}
\label{blowup}
\sum_{k=1}^{N} \frac{1}{k ||\alpha k||} \neq O(N^2)
\end{equation}
for $\alpha$ in a subset of second Baire category.
The construction uses the standard continued fractions expansion of the real numbers, explained in Lemma \ref{huxley}.

As shown in \cite{9}, we have $ || Q_n \alpha || < \frac{1}{Q_{n+1}}$ for any $n$, so that taking
some of the quantities $\frac{1}{k || \alpha k ||}$ in \eqref{blowup} to be large can be done by choosing sequences
$\frac{Q_{n+1}}{Q_n} > a_{n+1}$ increasing sufficiently rapidly.

We put on $\mathbb R \setminus \mathbb Q$ the topology induced by $\mathbb R$.
Given any finite sequence $\tilde{a}=(a_0;a_1,a_2,\dots,a_n)$, the set of numbers having $\tilde{a}$ as
their first partial quotients will form an open subset of $\mathbb R \setminus \mathbb Q$.
Also, a dense subset of $\mathbb R \setminus \mathbb Q$ must contain a number having $\tilde{a}$ as its
first partial quotients for any given $\tilde{a}$.

Since
\[
\sum_{k=1}^{N} \frac{1}{k ||\alpha k||} > \frac{1}{N || \alpha N ||}
\]
we will use sequences $(a_0;a_1,a_2,a_3,\dots)$ such that for any $C > 0$,
\[
\frac{Q_{n+1}}{Q_n} > a_{n+1} > C Q_n^2
\]
for some $n$, so that the bound $O(N^2)$ does not hold. However, $a_{n+1}$ has no dependence on $Q_n$ and there
are no restrictions to create such sequences.

We consider the following open dense subsets of $\mathbb R \setminus \mathbb Q$,
\[
S(C) = \{(a_0;a_1,a_2,a_3,\dots) | a_{n+1} > C Q_n^2 \text{ for some } n\}
\]
and the denumerable intersection $S^{*}=\cap_{k=1}^{+\infty} S(k)$. Any element of $S^{*}$, which is of second Baire category, will satisfy \eqref{blowup}.

\begin{remark}
It is mentioned in \cite{1} that the main results of this paper hold for a set of Liouville metrics of  second Baire's category, however no details of the proof of this statement are provided.
\end{remark}

\section{Tori of revolution}
\label{section4}
\subsection{Statement of results}
A torus of revolution $T= \mathbb R^2 / (a_1 \mathbb Z \oplus a_2 \mathbb Z)$ is a two-dimensional torus with the metric
\begin{equation}
ds^2 = U_1(q_1)(dq_1^2 + dq_2^2),
\end{equation}
where $U_1(q_1) > 0$ is a smooth periodic function on $\mathbb R$, satisfying $U_1(q_1+a_1) = U_1(q_1)$ for all $q_1 \in \mathbb R$.
For simplicity, we assume that $a_1=a_2=1$, but all the proofs work for arbitrary $a_i>0$.

\begin{definition}
\label{condition114}
Let $\Omega_{rev}$ be the set of functions $U_1$, satisfying the following conditions:
\begin{enumerate}
\item $U_1 \in C^{\infty}(\mathbb{R})$.
\item $U_1(q_1 + 1) = U_1(q_1)$ for all $q \in \mathbb R$.
\item The function $U_1$ has exactly one minimum and one maximum in $[0,1)$, both nondegenerate.
\end{enumerate}
\end{definition}

The Hamiltonian $H(p,q) = L^2$ and the first integral $S(p,q) = cL^2$ are given by putting $U_2(q_2)=0$ in section \ref{section21}.
The eigenvalue estimates found in section \ref{section23} still hold with $c_4=c_3=0$.
We keep the notations
\[
c_1 = \max_{0 \leq x \leq 1} U_1(x) = U_1(M_1), \qquad c_2 = \min_{0 \leq x \leq 1} U_1(x) = U_1(m_1),
\]
The curve $\gamma = (F_1(c),F_2(c))$ for $c \in [0,c_1]$ is now defined by
\[
F_1(c)=
 \int_{U_1(q_1) \geq c} (U_1(q_1) - c)^{1/2} dq_1
\]
\[
F_2(c)= c^{1/2}
\]
and its curvature is denoted by $\kappa(c)$.
One can show that for $\gamma=(t,f(t))$, the function $f$ satisfies the following asymptotics near $t=F_1(0)$,
\[
m |t-F_1(0)|^{-1/2} < |f'(t)| < M |t-F_1(0)|^{-1/2}
\]
\[
m |t-F_1(0)|^{-3/2} < |f''(t)| < M |t-F_1(0)|^{-3/2}
\]
for some $0<m<M$.

The definition of a nondegenerate metric of revolution reads as follows:

\begin{definition}
\label{genericity4}
The metric of revolution $ds^2=U_1(q_1)(dq_1^2 + dq_2^2)$ on $T$ is said to be nondegenerate if $U_1 \in \Omega_{rev}$
and the following conditions hold:
\begin{enumerate}
\item The curvature $\kappa(c)$ has a finite number of zeros on $(c_2,c_1]$, each of first order.
\item If $\kappa(\tilde{c})=0$, then ${F_2}'(\tilde{c})/{F_1}'(\tilde{c})$ is a typical number (see Definition \ref{typicalnum}).
\item The number ${F_2}' (c_1) / {F_1}' (c_1)$ is typical.
\[
{F_2}' (c_1) / {F_1}' (c_1) = \frac{  c_1^{-1/2} }
{-2\pi (-2U_{1}^{''}(M_1))^{-1/2}}
\]
\item The number $F_2 (c_2) / F_1 (c_2)$ is typical.
\[
F_2 (c_2) / F_1 (c_2) = \frac{  c_2^{1/2} }
{ \int_{0}^1 (U_1(q_1) - c_2)^{1/2} dq_1}
\]
\end{enumerate}
\end{definition}

\begin{theorem}
\label{thm14} The spectral counting function of a nondegenerate torus of revolution admits the following bound on its remainder term:
\[
R(\lambda) = O(\lambda^{2/3}).
\]
\end{theorem}

\begin{theorem}
\label{thm24} The set of nondegenerate metrics of revolution is dense in $\Omega_{rev}$ in the Whitney $C^{\infty}$--topology.
\end{theorem}

All the proofs of sections \ref{section31}, \ref{straight} and \ref{boundremterm} can be modified to treat the particular case of $U_2(q_2)=0$.
However the density property of the nondegenerate metrics of revolution has to be shown in another way.

\subsection{Nondegenerate metrics of revolution are dense}
The set of metrics satisfying the condition $(1)$ of Definition \ref{genericity4} is open and dense.
It is open since the curvature $\kappa(c)$ and its derivatives are continuous functions of $\Omega_{rev}$.
Any function $U_1$ can be approximated in $\Omega_{rev}$ by analytic functions, using its partial Fourier series.
Suppose that $\kappa(c_1) \neq 0$, by slightly modifying the analytic approximation of $U_1$ if needed.
We deduce that $\kappa(c)$ is analytic and has  a finite number of zeros in $(c_2,c_1)$, each of finite order. If some zeros are of order greater than one, we must perturb $U_1$ to
make them first order.
Remember that the curvature vanishes if and only if $(F_2^{''}F_1^{'} - F_1^{''}F_2^{'})(c)$ vanishes,
and that their zeros are of the same order.
Suppose $\kappa(\tilde{c})=0$ and the order of vanishing is greater than one.
It is sufficient to find $\tilde{U}_1$ such that
\[
\frac{d}{d \epsilon} \left. \kappa_{U_1+ \epsilon \tilde{U}_1}(\tilde{c}) \right|_{\epsilon=0}  \neq 0
\]
We have
\[
\frac{d}{d \epsilon} \left. (F_2^{''}F_1^{'} - F_1^{''}F_2^{'})(\tilde{c}) \right|_{\epsilon=0}
\]
\[
= -\frac{\tilde{c}^{-3/2}}{16} { \int_{U_1(q_1) \geq \tilde{c}} \frac{\tilde{U}_1 (q_1) dq_1}{(U_1(q_1) - \tilde{c})^{3/2}} }
-
\frac{3 \tilde{c}^{-1/2}}{16} { \int_{U_1(q_1) \geq \tilde{c}} \frac{\tilde{U}_1 (q_1) dq_1}{(U_1(q_1) - \tilde{c})^{5/2}} }
\]
\[
< 0
\]
if $\tilde{U}_1$ takes nonnegative values and has support sufficiently close to $M_1$.
Indeed, if the image of $\text{supp }(\tilde{U_1})$ under $U_1$ does not contain $\tilde{c}$, the
previous integrands will vanish near the ends of the interval of integration.
By taking $\epsilon$ small enough, the zeros of $\kappa_{U_1+ \epsilon \tilde{U}_1}$ in a neighbourhood
of $\tilde{c}$ will be of first order.
We deduce that the set of metrics for which $(1)$ holds in Definition \ref{genericity4} is also dense.

The last step of the proof requires to find a single perturbation $\tilde{U}_1$ of $U_1$ so that
the derivatives in $\epsilon$ of the functions in conditions $(2)$,$(3)$ and $(4)$ do not vanish.
We still suppose that the support of $\tilde{U}_1$
is concentrated around $M_1$, and need
\[
\int_{U_1(q_1) \geq \tilde{c}} \frac{\tilde{U}_1(q_1) dq_1}{(U_1(q_1)-\tilde{c})^{3/2}}
\neq 0
\]
at each zero $\tilde{c}$ of $\kappa$, for $(2)$.
In condition $(3)$,
\[
\frac{d}{d \epsilon} \left( c_1(\epsilon) (U_1+\epsilon \tilde{U}_1)''(M_1(\epsilon)) \right)
\]
does not vanish
if $\tilde{U}_1(q_1)$ can be written as $(q_1-M_1)^2\phi(q_1-M_1)$ with $\phi$ smooth and $\phi(0) \neq 0$.
Note that by definition $M_1(0)=M_1$, $c_1(0)=c_1$ and
\[
c_1(\epsilon) = \max_{0 \leq x \leq 1} (U_1+\epsilon \tilde{U}_1)(x) = (U_1+\epsilon \tilde{U}_1)(M_1(\epsilon))
\]
Finally, we need
\[
\int_0^1 \frac{\tilde{U}_1(q_1) dq_1}{(U_1(q_1)-c_2)^{1/2}}
\neq 0
\]
for $(4)$. We can obviously find $\tilde{U}_1$ satisfying all these conditions.

As in the proof of Theorem \ref{thm2},
there is an $\epsilon$, as small as required, for which $U_1 +\epsilon \tilde{U}_1$ satisfies simultaneously
conditions $(2)$,$(3)$ and $(4)$ of Definition \ref{genericity4}.

\section{Infra-Liouville tori}
\label{section24}
\subsection{Infra-Liouville metrics}
Assume that a Liouville torus $T$ admits a finite group of translations $G$ leaving the metric invariant, and consider $T/G$.
If $G$
is not of the form $G_1 \oplus G_2$, where $G_i$ is generated by translations along $q_i$, the induced metric
on $T/G$ will not be Liouville. Such metrics have been studied in \cite{3}, and are called infra-Liouville.
If  $T$ belongs to the conformal class of the square
flat torus $\mathbb{R}^2 / \mathbb{Z}^2$, then $G$ is spanned by $\langle (r_1,r_2),(s_1,s_2) \rangle$ with
$r_i,s_i \in \mathbb{Q}$, and there exists $(a_{i,j}) \in M_{2 \times 2}(\mathbb{Z})$ such that
\[
\left( \begin{array}{cccc} 1 & 0 \\ 0 & 1 \end{array} \right)
=
\left( \begin{array}{cccc} r_1 & s_1 \\ r_2 & s_2 \end{array} \right)
\left( \begin{array}{cccc} a_{11} & a_{12} \\ a_{21} & a_{22} \end{array} \right)
\]
The invariance of the metric implies
\[
\begin{cases}
U_i(q_i+r_i)-U(q_i)=v_i \in \mathbb{R}\\
U_i(q_i+s_i)-U(q_i)=w_i \in \mathbb{R}
\end{cases}
\]
and, since $U_i(q_i+1)=U_i(q_i)$,
\[
\left( \begin{array}{cccc} 0 & 0 \\ 0 & 0  \end{array} \right)
=
\left( \begin{array}{cccc}   v_1  &   w_1       \\   v_2    &   w_2     \end{array} \right)
\left( \begin{array}{cccc} a_{11} & a_{12} \\ a_{21} & a_{22} \end{array} \right)
\]
We conclude that $v_i=0$ and $w_i=0$ for $i=1,2$.
\subsection{Spectral properties of infra-Liouville tori}
Let $\frac{1}{n_i} = \inf \{z > 0 | z = x r_i + y s_i \text{ and } x,y \in \mathbb{Z} \} $, so that
$n_i \in \mathbb{N}$.
We deduce $U_i(q_i+\frac{1}{n_i})=U_i(q_i)$, $i=1,2$.
The eigenfunctions on $T/G$ can be written as products $\Psi_1(q_1)\Psi_2(q_2)$ for which
\begin{equation}
\label{sturmgener}
\begin{cases}
\Psi_1(q_1+\frac{1}{n_1}) = e^{i\frac{2\pi l_1 }{n_1}}\Psi_1(q_1)\\
\Psi_2(q_2+\frac{1}{n_2}) = e^{i\frac{2\pi l_2 }{n_2}}\Psi_2(q_2)
\end{cases}
\end{equation}
with $l_i \in \{0, \cdots, n_i-1\}$ and $(r_1 l_1 + r_2 l_2,s_1 l_1 + s_2 l_2) \in \mathbb{Z} \times \mathbb{Z}$.

If both $U_i$ have only one nondegenerate maximum and one nondegenerate minimum
on $[0,\frac{1}{n_i})$, the quantization rules found in Theorem 6.1 of \cite{1} can be generalized to
study solutions satisfying \eqref{sturmgener}. We can also associate a set of metrics $\Omega_G$ to $G$
and such pairs $(U_1,U_2)$. A metric in $\Omega_G$ is nondegenerate if the rescaled pair $\left(U_1 \left(\frac{q_1}{n_1} \right),U_2 \left(\frac{q_2}{n_2} \right) \right )$, which belongs to $\Omega$,
is nondegenerate according to Definition \ref{genericity}.
\begin{prop}
The remainder term of a nondegenerate infra-Liouville torus is of order $O(\lambda^{2/3})$,
and nondegenerate metrics are dense in $\Omega_G$.
\end{prop}

\textbf{Acknowledgements.} This research was conducted under the
supervision of Iosif Polterovich and supported by the NSERC Canada Graduate Scholarship. The problem was posed by Professor
Polterovich and I would like to thank him for his assistance.

\end{document}